\documentclass[12pt]{amsart}
\usepackage{ifthen,verbatim}
\usepackage{floatflt,graphicx,graphics}
\usepackage{tikz}
\usepackage{mathrsfs}
\usepackage{amsmath,amsfonts,amssymb,amsthm}
\usepackage{subcaption}

\nonstopmode
\numberwithin{equation}{section}
\theoremstyle{plain}
\newtheorem{theorem}[equation]{Theorem}

\newtheorem{lemma}[equation]{Lemma}
\newtheorem{corollary}[equation]{Corollary}

\theoremstyle{definition}

\newtheorem{remark}[equation]{Remark}

\theoremstyle{remark}

\newcommand{\R}{\mathbb{R}}

\newcommand{\B}{\mathbb{B}}

\newcommand{\uhp}{\mathbb{H}}

\newcounter{alphabet}

\newcounter{minutes}
\divide\time by 60
\newcounter{hours}
\multiply\time by 60
\addtocounter{minutes}{-\time}
\begin{document}
\bibliographystyle{amsplain}
\title
{Inequalities for a new hyperbolic type metric}

\def\thefootnote{}
\footnotetext{
\texttt{\tiny File:~\jobname .tex,
          printed: \number\year-\number\month-\number\day,
          \thehours.\ifnum\theminutes<10{0}\fi\theminutes}
}
\makeatletter\def\thefootnote{\@arabic\c@footnote}\makeatother

\author[O. Rainio]{Oona Rainio}

\begin{abstract}
We study a new hyperbolic type metric recently introduced by Song and Wang. We present formulas for it in the upper half-space and the unit ball domains and find its sharp inequalities with the hyperbolic metric and the triangular ratio metric. We also improve existing ball inclusion results and give bounds for the distortion of this new metric under conformal and quasiregular mappings.
\end{abstract}
\maketitle

\noindent Oona Rainio$^1$, email: \texttt{ormrai@utu.fi}, ORCID: 0000-0002-7775-7656,\newline 
1: University of Turku, FI-20014 Turku, Finland\\
\textbf{Funding.} My research was funded by Magnus Ehrnrooth Foundation and the Finnish Cultural Foundation.\\
\textbf{Acknowledgements.} I am thankful for Professor Matti Vuorinen for his constructive comments related to this work. %I would also like to thank the referee for their careful work.
\\
\textbf{Data availability statement.} Not applicable, no new data was generated.\\
\textbf{Conflict of interest statement.} There is no conflict of interest.

%%%%%%%%%%%%%%%%%%%%%%%%%%%%%%%%
%%%%%%%%%%%%%%%%%%%%%%%%%%%%%%%%
%%%%%%%%%%%%%%%%%%%%%%%%%%%%%%%%

\section{Introduction}

Due to is conformal invariance, the hyperbolic metric is a useful tool for studying different types of mappings in geometric function theory but this metric is not properly defined in higher dimensions outside the most simple domains of the upper half-space and the unit ball. For this reason, researchers have introduced several hyperbolic type metrics which share some of the properties of the hyperbolic metric but have a very easy definition \cite{g79,h,hkv,inm,z24}. The key feature is that the metrics measure the distance between two points by taking into account the position of the points with respect to each other and the boundary of their domain.

Typically, the hyperbolic metrics are defined by using a quotient where the numerator is the Euclidean distance of the two points and the denominator depends on their distance from the boundary. For instance, the distance ratio metric introduced by Gehring and Osgood \cite{g79} uses the minimum of the Euclidean distances from the two points to the boundary the as the denominator. However, this definition leads to one of the differences from the hyperbolic metric: If we rotate the point further away from the boundary around the other point so that the center of the rotation stays the point closer to the boundary, the value of the distance ratio metric does not change at all. Given this type of a transformation would affect the hyperbolic distance between the points, there is an interest in the hyperbolic type metrics that actually consider how the both points are positioned with respect to the boundary.

Recently, the following new hyperbolic type metric was introduced by Song and Wang \cite{s24}: In a domain $G\subsetneq\R^n$ whose boundary is denoted by $\partial G$, define the function $\tilde{c}_G:G\times G\to[0,\infty)$ as
\begin{align}
\tilde{c}_G(x,y)=\frac{|x-y|}{\inf_{z\in\partial G}\max\{|x-z|,|y-z|\}},\quad x,y\in G.    
\end{align}
Given the combination of infimum and maximum in the denominator, it is clear that this function is affected by the distance from the both points $x$ and $y$ to a same boundary point $z$. By \cite[Cor. 1]{s24}, this function $\tilde{c}_G$ is also truly a metric.

In this article, we continue the study of this metric. We give formulas for computing the value of this new metric in the upper half-space and the unit disk. We find the best possible constants for the inequalities between this new metric and both the hyperbolic metric and the triangular ratio metric. We use these inequalities formulate the ball inclusion results, which improve the ones found in \cite{s24}. We also present results about the distortion of the metric $\tilde{c}$ under conformal and quasiregular mappings. 

\section{Preliminaries}

For a point $x\in\R^n$, denote its $j$th coordinate by $x_j$ so that $x=(x_1,...,x_n)$. The upper half-space can now be defined as $\uhp^n=\{x\in\R^n\,:\,x_n>0\}$. Denote also the unit vectors of $\R^n$ by $e_1,...,e_n$. Denote the Euclidean open ball with a center $x\in\R^n$ and a radius $r>0$ by $B^n(x,r)$ and let $S^{n-1}(x,r)$ be its boundary sphere. Use the simplified notation $\B^n$ for the unit ball $B^n(0,1)$. Denote the line segment between $x,y\in\R^n$ by $[x,y]$. For two distinct points $x,y\in\R^n\setminus\{0\}$, let $\measuredangle XOY$ be the magnitude of the angle formed between the vectors from the origin to the points $x$ and $y$ so that $\measuredangle XOY\leq\pi$. 

By denoting the hyperbolic sine, cosine and tangent by sh, ch, and th, respectively, we can define the hyperbolic metric as \cite[(4.8), p. 52 \& (4.14), p. 55]{hkv}
\begin{align*}
\text{ch}\rho_{\uhp^n}(x,y)&=1+\frac{|x-y|^2}{2x_ny_n},\quad x,y\in\uhp^n,\\
\text{sh}^2\frac{\rho_{\B^n}(x,y)}{2}&=\frac{|x-y|^2}{(1-|x|^2)(1-|y|^2)},\quad x,y\in\B^n.
\end{align*}
From this, we can solve that
\begin{align}\label{for_hyphn}
{\rm th}\frac{\rho_{\uhp^n}(x,y)}{2}=\frac{|x-y|}{\sqrt{|x-y|^2+4x_ny_n}},
\quad
{\rm th}\frac{\rho_{\B^n}(x,y)}{2}=\frac{|x-y|}{\sqrt{|x-y|^2+(1-|x|^2)(1-|y|^2)}}.   
\end{align}
For a domain $G\subsetneq\R^n$, define then the distance ratio metric $j_G:G\times G\to[0,\infty)$ as \cite[p. 685]{chkv},
\begin{align}
j_G(x,y)=\log\left(1+\frac{|x-y|}{\min\{d_G(x),d_G(y)\}}\right)  
\end{align}
and the $j^*$-metric $j^*_G:G\times G\to[0,1],$ as \cite[2.2, p. 1123 \& Lemma 2.1, p. 1124]{hvz}
\begin{align}\label{def_j}
j^*_G(x,y)={\rm th}\frac{j_G(x,y)}{2}=\frac{|x-y|}{|x-y|+2\min\{d_G(x),d_G(y)\}}.    
\end{align}
The triangular ratio metric introduced by H\"ast\"o \cite{h} and studied in \cite{sch,fss,sinb} is defined as $s_G:G\times G\to[0,1],$ \cite[(1.1), p. 683]{chkv} 
\begin{align*}
s_G(x,y)=\frac{|x-y|}{\inf_{z\in\partial G}(|x-z|+|z-y|)}. 
\end{align*}

For a hyperbolic type metric $d\in\{\tilde{c},s,j,\rho\}$, define its balls in a domain $G\subsetneq\R^n$ so that $B_d(x,r)=\{y\in G\,:\,d_G(x,y)<r\}$ where $x\in\R^n$ and $r>0$.

\begin{theorem}\label{thm_sjbounds}
\emph{\cite[Lemma 2.1, p. 1124; Lemma 2.2, p. 1125; Lemma 2.8 \& Thm 2.9(1), p. 1129]{hvz}} For all $x,y\in G\subsetneq\R^n$, the inequality $j^*_G(x,y)\leq s_G(x,y)\leq2j^*_G(x,y)$ holds and, if $G$ is convex, then $s_G(x,y)\leq\sqrt{2}j^*_G(x,y)$.
\end{theorem}

\section{Results}

\begin{lemma}\label{lem_csG}
For all points $x,y$ in any domain $G\subsetneq\R^n$, the inequality $s_G(x,y)\leq\tilde{c}_G(x,y)\leq2s_G(x,y)$ holds and has the best constants regardless of the exact choice of $G\subsetneq\R^n$ and $x\in G$.    
\end{lemma}
\begin{proof}
Clearly,
\begin{align*}
&(1/2)\cdot\inf_{z\in\partial G}(|x-z|+|z-y|)
\leq1/2\cdot\inf_{z\in\partial G}(\max\{|x-z|,|y-z|\}+\max\{|x-z|,|y-z|\})\\
&=\inf_{z\in\partial G}\max\{|x-z|,|y-z|\}
\leq\inf_{z\in\partial G}(\max\{|x-z|,|y-z|\}+\min\{|x-z|,|y-z|\})\\
&=\inf_{z\in\partial G}(|x-z|+|z-y|),
\end{align*}
so the inequality follows. Let us then prove that there cannot be better constants for any choice of $G$. For $x\in G$, $u\in S^{n-1}(x,d_G(x))\cap\partial G$, and $y=u+k(x-u)$ with $0<k<1$,
\begin{align*}
\frac{\tilde{c}_G(x,y)}{s_G(x,y)}=\frac{|x-u|+|u-y|}{|x-u|}=1+k,    
\end{align*}
which approaches 1 when $k\to0^+$ and 2 when $k\to1^-$.
\end{proof}

\begin{corollary}\label{cor_csinc}
For any point $x$ in a domain $G\subsetneq\R^n$, the ball inclusion
\begin{align*}
B_s(x,R_0)\subseteq B_{\tilde{c}}(x,r)\subseteq B_s(x,R_1)  \end{align*}
holds for all $0<r<1$ if and only if $R_0\leq r/2$ and $R_1\geq r$.
\end{corollary}
\begin{proof}
By Lemma \ref{lem_csG}, we have 
\begin{align*}
B_s(x,r/2)=\{y\in G\,:\,s_G(x,y)<r/2\}\subseteq\{y\in G\,:\,\tilde{c}_G(x,y)<r\}=B_{\tilde{c}}(x,r)\subseteq B_s(x,r)    
\end{align*}  
and, since Lemma \ref{lem_csG} has the best possible constants for any choices of $G\subsetneq\R^n$ and $x\in G$, the radii $r/2$ and $r$ are the best possible ones here for an unspecified value of $0<r<1$.
\end{proof}

\begin{remark}
Corollary \ref{cor_csinc} improves the earlier inclusion result \cite[Cor. 2]{s24}.  
\end{remark}

\begin{lemma}\label{lem_cjG}
For all points $x,y$ in any domain $G\subsetneq\R^n$, the inequality $j^*_G(x,y)\leq\tilde{c}_G(x,y)\leq4j^*_G(x,y)$ holds. The inequality $j^*_G(x,y)\leq\tilde{c}_G(x,y)$ has always the best constant regardless of the exact choice of $G\subsetneq\R^n$ and $x\in G$, and the inequality $\tilde{c}_G(x,y)\leq4j^*_G(x,y)$ has the best constant in the case $G=\R^n\setminus\{0\}$. If $G$ is convex, then $\tilde{c}_G(x,y)\leq2\sqrt{2}j^*_G(x,y)$.    
\end{lemma}
\begin{proof}
The inequalities follow from Theorem \ref{thm_sjbounds} and Lemma \ref{lem_csG}. If $x\in G$, \newline $u\in S^{n-1}(x,d_G(x))\cap\partial G$, and $y=u+k(x-u)$ with $k\to0^+$,
\begin{align*}
\frac{\tilde{c}_G(x,y)}{s_G(x,y)}=\frac{|x-u|+|u-y|}{|x-u|}=1+k\to1.    
\end{align*}
For $x=e_1$ and $y=-e_1$,
\begin{align*}
\frac{\tilde{c}_{\R^n\setminus\{0\}}(x,y)}{j^*_{\R^n\setminus\{0\}}(x,y)}=\frac{2+2}{1}=4.    
\end{align*}
\end{proof}

\begin{theorem}\label{thm_chn}
For all $x,y\in\uhp^n$,
\begin{align*}
\tilde{c}_{\uhp^n}(x,y)=
\begin{cases}
\dfrac{|x-y|}{\max\{x_n,y_n\}}&\text{if}\quad|x-y|\leq\sqrt{(x_n-y_n)^2+|x_n^2-y_n^2|},\\
\vspace{1pt}\\
2\sqrt{\dfrac{|x-y|^2-(x_n-y_n)^2}{|x-y|^2+4x_ny_n}}&\text{otherwise.}    
\end{cases}
\end{align*}
\end{theorem}
\begin{proof}
Denote $x'=(x_1,...,x_{n-1},0)$ and $y'=(y_1,...,y_{n-1},0)$. In the special case $x'=y'$, the infimum $\inf_{z\in\partial\uhp^n}\max\{|x-z|,|y-z|\}$ is trivially found by choosing $z=x'=y'$ and we have $\tilde{c}_{\uhp^n}(x,y)=|x-y|/\max\{x_n,y_n\}$. In this case, we have 
\begin{align*}
|x-y|=|x_n-y_n|\leq\sqrt{(x_n-y_n)^2+|x_n^2-y_n^2|},    
\end{align*}
so the condition in the cases in the theorem is fulfilled.

Suppose then $x'\neq y'$. For any $u\in\partial\uhp^n\setminus[x',y']$, 
\begin{align*}
\max\{|x-v|,|y-v|\}<\max\{|x-u|,|y-u|\}    
\end{align*}
where $v$ is the closest point to $u$ on the line segment $[x',y']$. It follows from this that the infimum $\inf_{z\in\partial\uhp^n}\max\{|x-z|,|y-z|\}$ is obtained when $z\in[x',y']$ or, equivalently, $z=x'+k(y'-x')$ for some $0\leq k\leq1$. 

If $\max\{x_n,y_n\}\geq\min\{|x-y'|,|y-x'|\}$ or, equivalently, 
\begin{align}\label{ine_chn}
\max\{x_n,y_n\}\geq\sqrt{|x'-y'|^2+\min\{x_n,y_n\}^2}
\quad\Leftrightarrow\quad
|x-y|\leq\sqrt{(x_n-y_n)^2+|x_n^2-y_n^2|},
\end{align}
then the infimum $\inf_{z\in\partial\uhp^n}\max\{|x-z|,|y-z|\}$ is found at either $x'$ or $y'$ and we have $\tilde{c}_{\uhp^n}(x,y)=|x-y|/\max\{x_n,y_n\}$.

Suppose then the inequality \eqref{ine_chn} does not hold. Now, there is a value of $0<k<1$ such that $|x-z|=|y-z|$ with $z=x'+k(y'-x')$. Since $|x-z|$ is increasing with respect to $k$ and $|y-z|$ is decreasing with respect to $k$, the infimum $\inf_{z\in\partial\uhp^n}\max\{|x-z|,|y-z|\}$ is found at the point where this equality holds. We can solve that
\begin{align*}
&|x-z|=|y-z|
\quad\Leftrightarrow\quad
\sqrt{k^2|x'-y'|^2+x_n^2}=\sqrt{(1-k)^2|x'-y'|^2+y_n^2}\\
&\Leftrightarrow\quad
k=\frac{|x'-y'|^2-x_n^2+y_n^2}{2|x'-y'|^2}\\
&\Leftrightarrow\quad
|x-z|
=\frac{\sqrt{(|x'-y'|^2+x_n^2+y_n^2)^2-4x_n^2y_n^2}}{2|x'-y'|}
=\frac{|x-y|}{2}\sqrt{\frac{|x-y|^2+4x_ny_n}{|x-y|^2-(x_n-y_n)^2}}\\
&\Rightarrow\quad
\tilde{c}_{\uhp^n}(x,y)=2\sqrt{\frac{|x-y|^2-(x_n-y_n)^2}{|x-y|^2+4x_ny_n}}.
\end{align*}
The theorem follows.
\end{proof}

\begin{corollary}\label{cor_chyph}
For all $x,y\in\uhp^n$, the inequality
\begin{align*}
{\rm th}\frac{\rho_{\uhp^n}(x,y)}{2}\leq\tilde{c}_{\uhp^n}(x,y)\leq2\,{\rm th}\frac{\rho_{\uhp^n}(x,y)}{2}    
\end{align*}
holds and has the best possible constants.
\end{corollary}
\begin{proof}
Follows directly from Lemma \ref{lem_csG} because $s_{\uhp^n}(x,y)={\rm th}(\rho_{\uhp^n}(x,y)/2)$ by \cite[p. 460]{hkv}.  
\end{proof}

\begin{theorem}\label{thm_cinb}
For all $x,y\in\B^n$,
\begin{align*}
\tilde{c}_{\B^n}(x,y)=
\begin{cases}
\dfrac{|x-y|}{1-\min\{|x|,|y|\}}&\text{if}\quad\cos(\mu)\geq\dfrac{||x|^2-|y|^2|+2\min\{|x|,|y|\}}{2\max\{|x|,|y|\}},\\
\vspace{1pt}\\
\dfrac{|x-y|}{\sqrt{1+|x|^2-2|x|\cos(k)}}&\text{otherwise,} \end{cases}
\end{align*} 
where $\mu$ is the magnitude of the angle $\measuredangle XOY$ and $k\in(0,\mu)$ is chosen so that $|x|\cos(k)-|y|\cos(\mu-k)=(|x|^2-|y|^2)/2$.
\end{theorem}
\begin{proof}
In the special case where $x/|x|=y/|y|$, the infimum 
\begin{align*}
\inf_{z\in S^{n-1}(0,1)}\max\{|x-z|,|y-z|\}    
\end{align*}
is obtained when $z=x/|x|=y/|y|$ and it is equal to $1-\min\{|x|,|y|\}$. Suppose then $x/|x|\neq y/|y|$. Let $C$ be such an arc that belongs to an origin-centered circle, has $x/|x|$ and $y/|y|$ as its end points, and has a center angle with magnitude $\mu=\measuredangle XOY\in(0,\pi]$. For all $u\in S^{n-1}(0,1)\setminus C$, 
\begin{align*}
\max\{|x-v|,|y-v|\}<\max\{|x-u|,|y-v|\}    
\end{align*}
where $v$ is the closest point to $v$ on $C$. Consequently,
\begin{align}
&\inf_{z\in S^{n-1}(0,1)}\max\{|x-z|,|y-z|\}
=\inf_{z\in C}\max\{|x-z|,|y-z|\}\nonumber\\
&=\inf_{k\in[0,\mu]}\max\left\{\sqrt{1+|x|^2-2|x|\cos(k)},\sqrt{1+|y|^2-2|y|\cos(\mu-k)}\right\}\label{inf_cinb}.
\end{align}

Suppose first that $1-\min\{|x|,|y|\}$ is greater than or equal to $\min\{|x-z|,|y-z|\}$ for all points $z\in C$. Since $\sup_{z\in C}|x-z|=|x-y/|y||$, this is the case when
\begin{align}\label{ine_cinb0}
1-\min\{|x|,|y|\}\geq\min\{|x-y/|y||,|y-x/|x||\}.
\end{align}
We need to have either
\begin{align*}
1-\min\{|x|,|y|\}\geq\sqrt{1+\min\{|x|,|y|\}^2-2\min\{|x|,|y|\}\cos(\mu)}
\quad\Leftrightarrow\quad
\cos(\mu)\geq1 
\end{align*}
or
\begin{align}
&1-\min\{|x|,|y|\}\geq\sqrt{1+\max\{|x|,|y|\}^2-2\max\{|x|,|y|\}\cos(\mu)}\nonumber\\
&\Leftrightarrow\quad
\cos(\mu)\geq\frac{||x|^2-|y|^2|+2\min\{|x|,|y|\}}{2\max\{|x|,|y|\}}\label{ine_cinb1}.
\end{align}
The inequality \eqref{ine_cinb1} holds for $\cos(\mu)=1$, so the inequality \eqref{ine_cinb0} holds if and only if the inequality \eqref{ine_cinb1} holds. Since 
\begin{align*}
1-\min\{|x|,|y|\}\geq\sup_{z\in C}\min\{|x-z|,|y-z|\},    
\end{align*}
we have 
\begin{align*}
\inf_{z\in C}\max\{|x-z|,|y-z|\}=1-\min\{|x|,|y|\}.    
\end{align*}
While we have $\cos(\mu)<1$ as $x/|x|\neq y/|y|$, the infimum was the same in the case $x/|x|=y/|y|$ and $\cos(\mu)=1$ fulfills \eqref{ine_cinb1}, so first part of the theorem follows.  

Suppose then $1-\min\{|x|,|y|\}<\sup_{z\in C}\min\{|x-z|,|y-z|\}$. We have
\begin{align}\label{ine_sqrtc}
1-|x|<\sqrt{1+|y|^2-2|y|\cos(\mu)},\quad
1-|y|<\sqrt{1+|x|^2-2|x|\cos(\mu)},
\end{align}
so the infimum \eqref{inf_cinb} is obtained with such $k\in(0,\pi)$ that
\begin{align*}
&\sqrt{1+|x|^2-2|x|\cos(k)}=\sqrt{1+|y|^2-2|y|\cos(\mu-k)}\\
&\Leftrightarrow\quad
|x|\cos(k)-|y|\cos(\mu-k)=(|x|^2-|y|^2)/2.
\end{align*}
The theorem follows.
\end{proof}

\begin{theorem}\label{thm_chypb}
For all $x,y\in\uhp^n$, the inequality
\begin{align*}
{\rm th}\frac{\rho_{\B^n}(x,y)}{2}\leq\tilde{c}_{\B^n}(x,y)\leq2\,{\rm th}\frac{\rho_{\B^n}(x,y)}{2}    
\end{align*}
holds and has the best possible constants.
\end{theorem}
\begin{proof}
If the inequality \eqref{ine_cinb1} holds, then by Theorem \ref{thm_cinb}, we have
\begin{align*}
&{\rm th}(\rho_{\B^n}(x,y)/2)\leq\tilde{c}_{\B^n}(x,y)\\
&\Leftrightarrow\quad
1-\min\{|x|,|y|\}\leq\sqrt{|x-y|^2+(1-|x|^2)(1-|y|^2)}=\sqrt{1+|x|^2|y|^2-2|x||y|\cos(\mu)}\\
&\Leftrightarrow\quad
-2(1-\max\{|x|,|y|\}\cos(\mu))+\min\{|x|,|y|\}(1-\max\{|x|,|y|\}^2)\leq0\\
&\Leftarrow\quad
(1-\max\{|x|,|y|\})(-2+\min\{|x|,|y|\}+|x||y|)<0.
\end{align*}

Consider then the case where the inequality \eqref{ine_cinb1} does not hold.
The infimum \eqref{inf_cinb} is now obtained neither $k=0$ nor $k=\mu$. We need to show that
\begin{align*}
&{\rm th}(\rho_{\B^n}(x,y)/2)\leq\tilde{c}_{\B^n}(x,y)\\
&\Leftrightarrow\quad
\inf_{k\in(0,\mu)}\max\left\{\sqrt{1+|x|^2-2|x|\cos(k)},\sqrt{1+|y|^2-2|y|\cos(\mu-k)}\right\}\\
&\quad\quad\quad\leq\sqrt{|x-y|^2+(1-|x|^2)(1-|y|^2)}\\
&\Leftrightarrow\quad
\inf_{k\in(0,\mu)}\max\left\{|x|^2-2|x|\cos(k),|y|^2-2|y|\cos(\mu-k)\right\}\leq|x|^2|y|^2-2|x||y|\cos(\mu)
\end{align*}
The function $f_0(k)=|x|^2-2|x|\cos(k)$ is increasing with respect to $k$ on $(0,\mu)$ and $f_1(k)=|y|^2-2|y|\cos(\mu-k)$ is decreasing. Due to the inequalities \eqref{ine_sqrtc}, we know that these functions intersect for some $k\in(0,\mu)$ and this point of intersection is where the infimum above is obtained. Thus, to prove that this intersection point is less than or equal to $|x|^2|y|^2-2|x||y|\cos(\mu)$, we need to show that either $f_0(k)$
exceeds $|x|^2|y|^2-2|x||y|\cos(\mu)$ for a greater value of $k$ than $f_1(k)$ drops under $|x|^2|y|^2-2|x||y|\cos(\mu)$ or one of the functions $f_1,f_0$ is below $|x|^2|y|^2-2|x||y|\cos(\mu)$ for all $0<k<\mu$. Since we can solve
\begin{align*}
f_0(k)\leq|x|^2|y|^2-2|x||y|\cos(\mu)\quad\Leftrightarrow\quad
k\leq{\rm acos}\left(\frac{|x|-|x||y|^2+2|y|\cos(\mu)}{2}\right)
\end{align*}
and
\begin{align*}
f_1(k)\leq|x|^2|y|^2-2|x||y|\cos(\mu)\quad\Leftrightarrow\quad
k\geq\mu-{\rm acos}\left(\frac{|y|-|x|^2|y|+2|x|\cos(\mu)}{2}\right),
\end{align*}
we have to show that, for
\begin{align*}
u_0=|x|-|x||y|^2+2|y|\cos(\mu),\quad
u_1=|y|-|x|^2|y|+2|x|\cos(\mu),
\end{align*}
we have either
\begin{align*}
\mu-{\rm acos}(u_1/2)\leq{\rm acos}(u_0/2),\quad
u_0\geq2\cos(\mu),\quad\text{or}\quad
u_1\geq2\cos(\mu).
\end{align*}
Consequently, we need to prove that the left-most inequality above holds when $\cos(\mu)\leq\min\{u_0,u_1\}/2$ or, equivalently, when
\begin{align*}
\cos(\mu)\leq\frac{\min\{|x|,|y|\}+|x||y|}{2}.
\end{align*}
Given $\mu\leq\pi$, the inequality
\begin{align*}
\mu-{\rm acos}(u_1/2)\leq{\rm acos}(u_0/2)    
\end{align*}
holds trivially when
\begin{align*}
&{\rm acos}(u_0/2)+{\rm acos}(u_1/2)\geq\pi  
\quad\Leftrightarrow\quad
{\rm acos}(u_0/2)\geq\pi-{\rm acos}(u_1/2)={\rm acos}(-u_1/2)\\
&\Leftrightarrow\quad
u_0/2\leq -u_1/2
\quad\Leftrightarrow\quad
(|x|+|y|)(1-|x||y|+2\cos(\mu))\leq0\\
&\Leftrightarrow\quad
\cos(\mu)\leq-(1-|x||y|)/2.
\end{align*}
Suppose then $\cos(\mu)>-(1-|x||y|)/2$. We have now
\begin{align*}
&\mu-{\rm acos}(u_1/2)\leq{\rm acos}(u_0/2)
\quad\Leftrightarrow\quad
\cos({\rm acos}(u_0/2)+{\rm acos}(u_1/2))\leq\cos(\mu)\\
&\Leftrightarrow\quad
u_0u_1-\sqrt{4-u_0^2}\sqrt{4-u_1^2}-4\cos(\mu)\leq0.
\end{align*}
To prove the inequality above, it is enough to prove it holds when its left-side is at maximum. Since $u_0u_1-\sqrt{4-u_0^2}\sqrt{4-u_1^2}$ is clearly increasing with respect to both $u_0$ and $u_1$ and
\begin{align*}
&u_0<1-|y|^2+2|y|\cos(\mu)<\max\{1,2\cos(\mu)\},\\\
&u_1<1-|x|^2+2|x|\cos(\mu)<\max\{1,2\cos(\mu)\},
\end{align*}
we have
\begin{align*}
&u_0u_1-\sqrt{4-u_0^2}\sqrt{4-u_1^2}-4\cos(\mu)
<\max\{\-2-4\cos(\mu),8\cos^2(\mu)-4\cos(\mu)-4\}\\
&=(1+2\cos(\mu))\max\{-2,-(1-\cos(\mu))\}\leq0\\
&\Leftrightarrow\quad
\cos(\mu)\geq-1/2
\end{align*}
Since we assumed that $\cos(\mu)>-(1-|x||y|)/2$ and $-(1-|x||y|)/2\geq-1/2$, the inequality above holds within our bounds.

We have now proven that ${\rm th}(\rho_{\B^n}(x,y)/2)\leq\tilde{c}_{\B^n}(x,y)$ holds, regardless of whether \eqref{ine_cinb1} holds or not. The inequality $\tilde{c}_{\B^n}(x,y)\leq2{\rm th}(\rho_{\B^n}(x,y)/2)$ follows from Lemma \ref{lem_csG} because $s_{\B^n}(x,y)\leq{\rm th}(\rho_{\B^n}(x,y)/2)$ by \cite[p. 460]{hkv}. We can also show that our theorem has the best constants: For $x=ue_1$ and $y=ve_1$ with $0<u<v<1$,
\begin{align*}
\frac{\tilde{c}_{\B^n}(x,y)}{{\rm th}(\rho_{\B^n}(x,y)/2)}
=\frac{\sqrt{1-2uv+u^2v^2}}{1-u}
=\frac{1-uv}{1-u},
\end{align*}
and, if $v\to u^+$, this quotient becomes $1+u$, which approaches 0 when $u\to0^+$ and 1 when $u\to1^-$.
\end{proof}

\begin{corollary}\label{cor_chypinc}
For all $0<r<1$ and $x\in G\in\{\uhp^n,\B^n\}$, the ball inclusion
\begin{align*}
B_\rho(x,R_0)\subseteq B_{\tilde{c}}(x,r)\subseteq B_\rho(x,R_1)    
\end{align*}
holds if and only if $R_0\leq\log\left(\dfrac{2+r}{2-r}\right)$ and $R_1\geq\log\left(\dfrac{1+r}{1-r}\right)$.
\end{corollary}
\begin{proof}
By Corollary \ref{cor_chyph} and Theorem \ref{thm_chypb},
\begin{align*}
&B_\rho\left(x,\log\left(\frac{2+r}{2-r}\right)\right)
=\left\{y\in G\,:\,{\rm th}\frac{\rho_G(x,y)}{2}\leq r/2\right\}
\subseteq B_{\tilde{c}}(x,r)\\
&\subseteq\left\{y\in G\,:\,{\rm th}\frac{\rho_G(x,y)}{2}\leq r\right\}=B_\rho\left(x,\log\left(\frac{1+r}{1-r}\right)\right)
\end{align*}
for $G\in\{\uhp^n,\B^n\}$ and, since Corollary \ref{cor_chyph}, and Theorem \ref{thm_chypb} have the best possible constants in their inequalities, we also have the best bounds for $R_0$ and $R_1$ here.
\end{proof}

\begin{remark}
Corollary \ref{cor_chypinc} improves the earlier inclusion result \cite[Thm 5]{s24}.  
\end{remark}

\begin{corollary}
If $f:G\to f(G)$ is a conformal mapping between domains $G,f(G)\in\{\uhp^n,\B^n\}$,
\begin{align*}
\tilde{c}_{f(G)}(f(x),f(y))\leq2\,\tilde{c}_G(x,y).    
\end{align*}
\end{corollary}
\begin{proof}
By Corollary \ref{cor_chyph} and Theorem \ref{thm_chypb} and the conformal invariance of the hyperbolic metric,
\begin{align*}
\tilde{c}_{f(G)}(f(x),f(y))\leq2\,{\rm th}\frac{\rho_{f(G)}(f(x),f(y))}{2}=2\,{\rm th}\frac{\rho_G(x,y)}{2}\leq2\,\tilde{c}_G(x,y).    
\end{align*}    
\end{proof}

\begin{corollary}
If $f:G\to f(G)$ is a non-constant quasiregular mapping between domains $G,f(G)\in\{\uhp^n,\B^n\}$,
\begin{align*}
\tilde{c}_{f(G)}(f(x),f(y))\leq2\,\lambda_n^{1-\alpha}(\tilde{c}_G(x,y))^\alpha,    
\end{align*}  
where $\alpha=K_I(f)^{1/(1-n)}$ with the inner dilatation $K_I(f)\leq K$ of the mapping $f$, and $\lambda_n$ is a constant defined in \cite[(9.6), p. 158]{hkv} that fulfills $\lambda_2=4$ and $\lambda_n<2e^{n-1}$ for $n\geq2$.
\end{corollary}
\begin{proof}
Follows by Corollary \ref{cor_chyph}, Theorem \ref{thm_chypb}, and \cite[Thm 16.2(1), p. 300]{hkv}.    
\end{proof}

\bibliographystyle{siamplain}
%\bibliography{Bref}

\end{document}